\documentclass{amsart}[12pt]
\usepackage{epsfig,graphicx,young,amsmath, amssymb, amsfonts,fancyhdr,hyperref}
\usepackage{lipsum}
\usepackage{lineno}

\fancypagestyle{titlepage}{
\fancyhead[R]{}
\fancyhead[CL]{}
\cfoot{\thepage}
\footskip=20pt
}

\def\ack{\section*{\bf Acknowledgements}}

\newtheoremstyle{theorem}
  {12pt}		  
  {0pt}  
  {\sl}  
  {\parindent}     
  {\bf}  
  {. }    
  { }    
  {}     
\theoremstyle{theorem}
\newtheorem{thm}{Theorem}  
\newtheorem{lemma}[thm]{Lemma}     

\theoremstyle{definition}

\newcommand{\lam}{\lambda}
\newcommand{\cA}{\mathcal{A}}
\newcommand{\cB}{\mathcal{B}}

\newcommand{\Section}[1]{\section{\bf #1}}

\makeatletter
\renewcommand*\@maketitle{%
  \normalfont\normalsize
  \@adminfootnotes
  \@mkboth{\@nx\shortauthors}{\@nx\shorttitle}%
  \global\topskip42\p@\relax 
  \@settitle
  \ifx\@empty\authors \else {\vskip 1em \vtop{\centering\shortauthors\@@par}} \fi
  \ifx\@empty\@date \else {\vskip 1em \vtop{\centering\@date\@@par}}\fi
  \ifx\@empty\@dedicatory
  \else
    \baselineskip18\p@
    \vtop{\centering{\footnotesize\itshape\@dedicatory\@@par}%
      \global\dimen@i\prevdepth}\prevdepth\dimen@i
  \fi
  \@setabstract
  \normalsize
  \if@titlepage
    \newpage
  \else
    \dimen@34\p@ \advance\dimen@-\baselineskip
    \vskip\dimen@\relax
  \fi
} 
\renewcommand*\@adminfootnotes{%
  \let\@makefnmark\relax  \let\@thefnmark\relax
  \ifx\@empty\@subjclass\else \@footnotetext{\@setsubjclass}\fi
  \ifx\@empty\@keywords\else \@footnotetext{\@setkeywords}\fi
  \ifx\@empty\thankses\else \@footnotetext{%
    \def\par{\let\par\@par}\@setthanks}%
  \fi
}

\makeatother

\begin{document}

\title{A Sidon-type condition on set systems}

\author{Peter J.~Dukes and Jane Wodlinger}
\address{\rm
Mathematics and Statistics, 
University of Victoria, Victoria, Canada
}
\email{dukes@uvic.ca, jw@uvic.ca}

\thanks{Ths research is supported by NSERC}


\date{\today}

\begin{abstract}
Consider families of $k$-subsets (or blocks) on a ground set of size $v$.  Recall that if all $t$-subsets occur with 
the same frequency $\lambda$, one obtains a $t$-design with index $\lam$.  
On the other hand, if all $t$-subsets occur with different frequencies, such a family has been called (by Sarvate and others) a $t$-adesign.  
An elementary observation shows that such families always exist for $v > k \ge t$.  Here, we study the smallest possible maximum frequency $\mu=\mu(t,k,v)$.

The exact value of $\mu$ is noted for $t=1$ and an upper bound (best possible up to a constant multiple) is obtained for $t=2$ using PBD closure.  Weaker, yet still reasonable asymptotic bounds on $\mu$ for higher $t$ follow from a probabilistic argument.  Some connections are made with the famous Sidon problem of additive number theory.
\end{abstract}

\maketitle
\thispagestyle{titlepage}

\hrule

\Section{Introduction}

Given a family (which may contain repetition) $\cA$ of subsets of a ground set $X$, the {\em frequency} of a set $T \subset X$ 
is the number of elements of $\cA$ (counting multiplicity) which contain $T$.  

Let $v \ge k \ge t$ be nonnegative integers.  A $t$-\emph{design}, or $S_\lam(t,k,v)$,  is a pair $(V,\cB)$ where $\cB$ is a family of $k$-subsets of $V$ such that every $t$-subset has the same frequency $\lam$.  Typically, $V$ is called a set of \emph{points}, $\cB$ are the \emph{blocks}, $t$ is the \emph{strength} (reflecting that $t$-designs are also $i$-designs for $i \le t$) and $\lam$ is the \emph{index}.  Repeated blocks are normally permitted in the definition.

There are `divisibility' restrictions on the parameters $v,k,t,\lam$ and beyond that very little is known in general about the existence of $S_\lam(t,k,v)$.  There are some trivial cases, such as $t=0$, $t=k$ or $k=v$, and some mildly interesting ones: $\lam=t=1$ leads to uniform partitions; $\lam=\binom{v-t}{k-t}$ is realized via the complete $k$-uniform hypergraph of order $v$.  For $t=2$ and fixed $k$ there is a rich and deep asymptotic existence theory due to R.M.~Wilson; see \cite{W}.  Spherical geometries and Hadamard matrices lead to some examples for $t=3$.

In \cite{SB1}, Sarvate and Beam consider an interesting twist on the definition.  A $t$-{\em adesign} is defined as a pair $(V,\cA)$, where $V$ is a ground set of $v$ 
points and $\cA$ is a collection of blocks of size $k$, satisfying the 
condition that every $t$-subset of points has a {\bf different} frequency.

Here, we abbreviate a $t$-adesign with $A(t,k,v)$.  It is easy to 
see that such families always exist for integers $v>k \ge t \ge 1$: simply assign multiplicities to $\binom{V}{k}$ according to different powers of two.

This begs a more intricate question.
Let $\mu(t,k,v)$ denote the smallest maximum frequency, taken over all 
adesigns $A(t,k,v)$.  The main question motivating this article is the 
following.

{\bf Problem.}
Given $t,k,v$, determine (or bound) $\mu(t,k,v)$.

In most of the previous investigations on adesigns, the cases of interest have been for $t=2$ and when the different pairwise frequencies are $1,2,\dots,\binom{v}{2}$.  It should be noted that here we allow zero as a frequency, although if desired it is not hard to bump up all frequencies to be positive.

From the definitions and easy observations above, we have
\begin{equation}
\label{basic-ineq}
\binom{v}{t}-1 \le \mu(t,k,v) < 2^{\binom{v}{k}}.
\end{equation}
However, the basic upper bound in (\ref{basic-ineq}) is unsatisfactory, at least asymptotically in $v$.  Our main goal is a substantial reduction of the upper bound (to something independent of $k$).

\begin{thm}
\label{main-t}
For $k > 2t+2$ and sufficiently large $v$, $$\mu(t,k,v) \le 16 t v^{2t+2} \log v.$$
\end{thm}

The constant is surely not best possible; however, we are content until more is known about the exponent.  

We can do much better when $t \le 2$.  For $t=1$, an elementary argument gives the exact value of $\mu$.  And for $t=2$, Wilson's theory of PBD closure reduces the upper bound on $\mu$ to a constant multiple of its lower bound.

\begin{thm}
For positive integers $v>k$,
\label{main-1}
$$\mu(1,k,v)=
\begin{cases}
v-1 & \text{if}~ 2k \le v ~\text{and}~\binom{v}{2} \equiv 0 \pmod{k}, \cr
v & \text{if}~ 2k \le v ~\text{and}~\binom{v}{2} \not \equiv 0 \pmod{k}, \cr
\left\lceil \frac{1}{v-k} \binom{v}{2} \right\rceil & \text{if}~ 2k>v.
\end{cases}$$
\end{thm}

\begin{thm}
\label{main-2}
There is a constant $C=C(k)$ such that 
$\mu(2,k,v) \le C v^2.$
\end{thm}

The proof of Theorem~\ref{main-t} follows a probabilistic argument and occurs in Section 2.  The proofs of Theorems~\ref{main-1} and \ref{main-2} are given in Section 3.  

Before beginning our detailed investigations, we should mention some connections with a central topic in additive combinatorics.  Briefly, a {\em Sidon sequence} (or \emph{Golomb ruler}) is a list of positive integers whose pairwise sums are all distinct, up to swapping summands.  More generally, a $B_r$-\emph{sequence}
or \emph{Sidon sequence of order} $r$ has the property that all its $r$-wise sums are distinct.
It is known (see \cite{Sidon-survey}) that the largest cardinality $F_r(n)$ of a Sidon sequence of order $r$ contained in $[n]$ satisfies
\begin{equation}
\label{sidon-bounds}
n^{1/r} (1-o(1)) \le F_r(n) \le  C(r) n^{1/r}.
\end{equation}  

Now consider an adesign $A(t,v-1,v)$, where $V$ is the ground set of size $v=k+1$.  
Assign multiplicity $f(x)$, chosen from a Sidon sequence of order $t$, to the `co-singleton' set $V \setminus \{x\}$, $x \in V$.  The inherited weight on a $t$-subset $T$ is $\sum_{x \not\in T} f(x)$.  By construction, this takes distinct values on all $t$-subsets.  From this and (\ref{sidon-bounds}), we see that $\mu(t,v-1,v) \le C v^t$, which is best possible up to a constant multiple.
However, it is also clear that the exact determination of $\mu$, even in the case $v=k+1$, is as difficult as the Sidon problem.

\Section{The general bound}

We prove Theorem~\ref{main-t} by employing $B_r$-sequences along the lines of the discussion concluding Section 1.  But here, a probabilistic selection is needed to control the upper bound on $\mu$.

\begin{proof}[Proof of Theorem~\ref{main-t}]
Assume $t > 1$, appealing to Theorem~\ref{main-1}.  Suppose first that $v$ is a prime power.  Bose and Chowla \cite{BC} construct a $B_t$-sequence of size $v$ in $[v^t]$.  Let $V$ be such a set of integers 
and consider the family $\cB$ of all $k$-subsets of $V$, where a $k$-set $K$ is taken with multiplicity
$$f(K) = \sum_{m \in V \setminus K} m.$$  Then the frequency of a $t$-subset $T$ in $\cB$ is 
\begin{equation}
\label{freq-tsub}
f(T) = \sum_{K \supseteq T, |K|=k} f(K) = \binom{v-t-1}{k-t} \sum_{m \in V \setminus T} m.
\end{equation}
By choice of $V$, these are all distinct frequencies.  Observe that $\sum_{m \in V \setminus T} m < v^t(v-t)$, so that $f(T)$ is at most a polynomial of order $v^{k+1}$.

Consider next a family $\cA=\cA(p)$ consisting of each element of $\cB$ chosen independently with probability $p$.  We claim there is some $p$ guaranteeing an adesign $A(t,k,v)$ of the required form.

Let $f_\cA(T)$ denote the frequency of $T$ in $\cA$.  This is a sum of $f(T)$ independent binomial random variables $X_i$, one for each $k$-set in $\cB$ containing $T$.  So $f_\cA(T)$ has expected value $\mu = p f(T)$ by linearity.  

Now let's invoke a (weak but tidy) two-sided Chernoff bound of the form
$$\mathbb{P}\left[ \left|\sum X_i - \mu \right|> 2 \sqrt{\mu \log 1/\epsilon} \right] < \epsilon,$$
which holds for $\epsilon > \exp(-\mu/4)$. 
Taking $\epsilon = \binom{v}{t}^{-1}$, we conclude that there exists (with positive probability) a family $\cA$ such that
\begin{equation}
\label{pinned-freqs}
|f_\cA(T)-pf(T)| < \sigma(T),
\end{equation}
for {\bf every} $t$-set $T$, where $\sigma(T):=2 \sqrt{p f(T) \log \binom{v}{t}}$, and for each $p$ with $p f(T) > 4 \log \binom{v}{t}$.

It remains to check that frequencies $f_\cA(T)$ remain distinct and appropriately bounded for some choice of $p$.  
By (\ref{freq-tsub}) and (\ref{pinned-freqs}), we have distinct frequencies provided that
$$2 \sigma(T) < p \binom{v-t-1}{k-t}.$$
Using the definition of $\sigma(T)$ and $\sum_{m \not\in T} m < v^{t} (v-t)$, it suffices to have
\begin{equation}
\label{est-p}
16 v^{t} (v-t) \log  \binom{v}{t}  < p \binom{v-t-1}{k-t}.
\end{equation}
The right side of (\ref{est-p}) grows faster than the left for $k \ge 2t+2$; hence, for sufficiently large $v$, we can choose $p = 16 v^{t+1} \log \binom{v}{t}/\binom{v-t-1}{k-t}<1$ (easily permitting application of the Chernoff bound above.)
 
For such a choice, we have
$$\max_T f_\cA(T) < p f(T) +  \sigma(T) < (16 v^{2t+2} +8 v^{t+1}) \log \binom{v}{t}.$$
The bound $\log \binom{v}{t} \le t \log v - \log t!$ leaves enough room to eliminate the lower-order term and imply the stated bound.

Finally, if $v$ is not a prime power, we can simply apply the above argument to a prime $v' \le v+o(v)$ to obtain asymptotically the same result.
\end{proof}

\Section{The cases $t=1$ and $t=2$}

When $t=1$, we  simply demand that every point is in a different number of blocks.  A complete characterization is possible here, following a technique known to Sarvate and Beam in early investigations.  To the best of our knowledge, though, Theorem~\ref{main-1} has not been worked out for general $v$ and $k$.

The proof strategy is as follows.  Suppose $f(1)<\cdots<f(v)$ are desired pointwise frequencies whose sum $F$ is divisible by $k$.  Set up $b=F/k$ blocks, and place element `1' in the first $f(1)$ blocks, element `2' in the next $f(2)$ blocks, and so on, with blocks identified modulo $b$.  In other words, the $i$th block contains those elements $x$ such that
$$\sum_{1 \le y< x} f(y) < bq+i \le \sum_{1 \le y \le x} f(y)$$
for some integer $q \in \{0,1,\dots,k-1\}$.  Care must be taken that the maximum frequency $f(v)$ does not exceed $b$, the number of blocks.  Ideally, the frequencies are chosen to be consecutive, or almost consecutive integers.

\begin{proof}[Proof of Theorem~\ref{main-1}]
We apply the above construction using a run of (almost) consecutive prescribed frequencies.  There is a division into two main cases.

\noindent
{\sc Case 1}: $2k \le v$.  

Suppose first that $k \mid \binom{v}{2}$.
Fill $b=\binom{v}{2}/k$ blocks with pointwise frequencies $0,1,\dots,v-1$.  Note that $b \ge v-1$ follows from the assumption $2k \le v$.  On the other hand, if $\binom{v}{2} = bk-r$, $0<r<k$, use $b$ blocks with frequencies $0,1,\dots,v-r-1,v-r+1,\dots,v$.  One has sum of frequencies $bk=\binom{v+1}{2}-(v-r)=\binom{v}{2}+r$, as required.  In either sub-case, the smallest possible maximum frequency is realized and we have
$$\mu(1,k,v) = 
\begin{cases}
v-1 & \text{if}~ \binom{v}{2} \equiv 0 \pmod{k}, \cr
v & \text{if}~ \binom{v}{2} \not \equiv 0 \pmod{k}.
\end{cases}
$$
\noindent
{\sc Case 2}:  $2k > v$.  We first show that the given value $\left\lceil \frac{1}{v-k} \binom{v}{2} \right\rceil$ is a lower bound on $\mu(1,k,v)$.  Suppose $m$ is the maximum frequency in an adesign $A(1,k,v)$.  Then 
$$mk \le bk \le (m-v+1)+\dots+(m-1)+m = mv-\binom{v}{2}.$$
In other words, $m$ is an integer with $m(v-k) \ge \binom{v}{2}$ and the lower bound follows.
Conversely, we must realize the given value $\mu:=\lceil \frac{1}{v-k} \binom{v}{2}\rceil$ as the maximum frequency in an adesign $A(1,k,v)$.  Put $bk=\mu v-\binom{v}{2}-r$, for some positive integer $b$ and $0 \le r < k$.
Again, use the strategy preceding the statement of the theorem, with 
$b=\frac{1}{k}(\mu v-\binom{v}{2}-r)$ blocks and frequencies 
$$\mu-v,\dots,\mu-v-r-1,\mu-v-r+1,\dots,\mu.$$  
It remains to check that $\mu \le b$.  However, this follows easily since $\mu$ is the least integer with 
$\mu(v-k) \ge \binom{v}{2}$.  Therefore, $\mu k \le \mu v-\binom{v}{2}$.  On the other hand, $b$ is the greatest integer so that $bk \le \mu v-\binom{v}{2}$. 
\end{proof}

We turn now to adesigns with $t=2$.  An important tool here is `PBD closure', which we briefly outline.  Let $K$ be a set of positive integers, each at least two.  A {\em pairwise balanced design} PBD$(v,K)$ is a set of $v$ points, together with a set of blocks whose sizes are in $K$, having the property that every unordered pair of different points is contained in exactly one block.  Wilson's theorem \cite{W} asserts that the necessary `global' and `local' divisibility conditions on $v$ given $K$ are asymptotically sufficient for the existence of PBD$(v,K)$.

A key observation for the proof of Wilson's theorem is the `breaking up blocks' construction: a block, say of size $u$, of a PBD can be replaced with the family of blocks of a PBD on $u$ points.  In particular, if there exists a PBD$(v,K)$ and an $S_\lam(2,k,u)$ for every $u \in K$, then there exists an $S_\lam(2,k,v)$.  

It was observed in \cite{DSG} that adesigns actually obey a similar recursion.  The basic idea is to place adesigns (instead of designs) on the blocks of a PBD.  However, each such adesign needs to be accompanied with a block design on those points with sufficiently large $\lam$ so as to `spread out' the pairwise frequencies.  When restated using $\mu$, one obtains the following result.

\begin{lemma}
\label{pbd-cl}
Suppose there exists a PBD$(v,K)$ with $b$ blocks having sizes $u_1$, $u_2$, $\dots$, $u_b$.
Put $M_0=0$ and for $0<i\le b$,
\begin{equation}
\label{spread}
M_i = \min\{\lam \ge M_{i-1} : \exists S_\lam(2,k,u_i) \} + \mu(2,k,u_i).
\end{equation}
Then $\mu(2,k,v) \le M_b$.
\end{lemma}

\noindent
\emph{Remark}.  The minimum in (\ref{spread}) is well defined; more generally, $S_\lam(t,k,v)$ exists for a smallest positive integer $\lam=\lam_{\rm min}(v,k) \le \binom{v-t}{k-t}$, and such designs can be repeated with arbitrary multiplicity.

We are now ready to prove the quadratic upper bound on $\mu(2,k,v)$.

\begin{proof}[Proof of Theorem~\ref{main-2}]
For $v$ large and $K=\{k+1,k+2,k+3\}$, apply Lemma~\ref{pbd-cl} to a PBD$(v,K)$.
Note that such PBDs exist for all sufficiently large integers $v$.  This follows easily from Wilson's theorem since the three consecutive block sizes lead to no `divisibility' restrictions on $v$ (globally, $\gcd\{k(k+1),(k+1)(k+2),(k+2)(k+3)\} = 2$ always divides $\binom{v}{2}$; locally, $\gcd\{k,k+1,k+2\}=1$ divides $v-1$).

Put $m = \max\{\mu(2,k,k+j):j=1,2,3\}$ and $l = \max\{\lam_{\rm min}(k+j,k):j=1,2,3\}$.  Observe $l$ and $m$ depend only on $k$.  Also, observe that the number of blocks $b$ of a PBD$(v,K)$ satisfies $b \le \binom{v}{2}/\binom{k+1}{2}$, since $k+1$ is the smallest block size.  Combining these facts, it follows that $$\mu(2,k,v) \le lmb \le C(k) \binom{v}{2}.\qedhere$$ 
\end{proof}

\Section{Discussion}

There is another noteworthy construction of $t$-adesigns by combining copies of systems which are nearly designs.  The general idea to work from a family $\widehat{\cB}_T$ of $k$-subsets such that one preferred $t$-subset $T$ has frequency $\lam_1$ and all other $t$-subsets have frequency $\lam_2 < \lam_1$.  (Such families can be found, for instance, via a linear algebraic argument upon `clearing denominators'.)  Then, take copies of $\widehat{\cB}_T$ with distinct multiplicities over each $T \in \binom{V}{t}$.  The crude bound obtained in this way is $\mu(t,k,v) \le C_1 \lam_1 v^t + C_2 \lam_2 v^{2t}$.  However, we presently see no way of keeping $\lam_2$ small enough in general.  This would be an interesting problem in its own right.  When such families $\widehat{\cB}_T$ exist with reasonable $\lam_2$, it is possible to improve Theorem~\ref{main-t}.

The remaining work for $t=2$ essentially amounts to a reduction in the multiplicative constant in Theorem~\ref{main-2}.  There are some ideas which seem promising in this direction.  For instance, $\mu(2,3,v)$  was completely determined in \cite{DSG} using a blend of  PBD closure, group divisible designs, and a variation on `anti-magic cubes'.  The latter concerns a neat side-problem: place nonnegative integers in the cells of the cube $[n]^3$ so that the $3n^2$ line sums are distinct and with maximum value as small as possible.  Interesting constructions yielding line sums $\{0,1,\dots,3n^2-1\}$ were found for $n=2,3,5,7$, and products of these values.

Returning to $\mu(2,3,v)$, a slightly technical argument shows that the maximum frequency for triples versus pairs in an adesign is best possible.

\begin{thm}[\cite{DSG}]
For all $v>3$, $$\mu(2,3,v) =\begin{cases}
\binom{v}{2},~\text{if}~v=4~\text{or}~v \equiv 2 \pmod{3},\\
\binom{v}{2}-1,~\text{otherwise}.
\end{cases}$$
\end{thm}

We omit further analysis of $\mu(2,k,v)$ for $k>3$ until better general constructions surface for $v$ small relative to $k$.  For fixed $k$, the complete determination of $\mu(2,k,v)$ can probably be reduced to a finite problem.  Quite possibly $\mu(2,k,v) = \binom{v}{2}-1+o(v)$ for each $k$.

\ack
The authors are indebted to a referee who suggested the probabilistic method, yielding the present form of Theorem~\ref{main-t}.

\end{document}